\documentclass[a4paper,12pt]{elsarticle}

\usepackage[utf8]{inputenc}
\usepackage[english]{babel}
\usepackage{hyperref}
\usepackage{amsmath,amsthm,enumerate}
\usepackage{amssymb}
\usepackage{graphicx}
\usepackage{tikz}
\usetikzlibrary{calc,decorations.pathmorphing,decorations.text,fixedpointarithmetic}
\usepackage{subfigure}
\usepackage{refcount}
\usepackage[hmargin=2.5cm,vmargin=3cm]{geometry}
\usepackage{mathtools, braket}

\newtheorem{theorem}{Theorem}
\newtheorem{proposition}[theorem]{Proposition}
\newtheorem{lemma}[theorem]{Lemma}

\newtheorem{corollary}[theorem]{Corollary}
\theoremstyle{definition}	

\newtheorem{definition}[theorem]{Definition}

\newtheorem{conjecture}[theorem]{Conjecture}

\newtheorem{problem}[theorem]{Problem}

\newcommand\inv{^{-1}}

\newcommand\cC{\mathcal{C}}

\newcommand\hG{\widehat{G}}

\newcommand{\cS}{{\mathcal{S}}}

\newcommand\cW{\mathcal{W}}

\newcommand\SPC{\operatorname{SPC}}
\newcommand\EDC{\operatorname{EDC}}
\newcommand\id{\mathrm{id}}

\usepackage[color=green!30]{todonotes}

\usepackage{graphicx,color}    
\usepackage{epsfig}

\usepackage{caption}
\usepackage{tkz-graph}
\usetikzlibrary{matrix}
\usetikzlibrary{decorations.markings}
\tikzstyle{vertex}=[circle, draw, fill=black!50,
                        inner sep=0pt, minimum width=4pt]
\tikzset{->-/.style={decoration={
  markings,
  mark=at position .5 with {\arrow{>}}},postaction={decorate}}}

\tikzstyle{bigblue}=[color=blue, very thick, >=stealth]
\tikzstyle{lightblue}=[color=blue, thin, >=stealth]

\tikzstyle{bigred}=[color=red, very thick, >=stealth]
\tikzstyle{lightred}=[color=red, thin, >=stealth]

\tikzstyle{biggreen}=[color=black!30!green, very thick, >=stealth]
\tikzstyle{lightgreen}=[color=black!30!green,  thin, >=stealth]

\makeatletter

\newcommand{\preclosure}[1]{%
	\ThisStyle{%
		\vbox {\m@th\ialign{##\crcr
				\preclosurefill \crcr
				\noalign{\kern-\p@\nointerlineskip}
				$\hfil\SavedStyle{#1}\hfil$\crcr}}}}

\def\preclosurefill{%
	$\m@th%
	\xleaders\hbox{$\mkern0mu\shortbar\mkern0mu$}\hfill%
	\shortbar%
	$}

\def\shortbar{%
	\smash{\scalebox{0.3}[1.0]{$-$}}}
\makeatother

\newcommand\ERASE[1]{}

\begin{document}
\begin{frontmatter}
	
\title{{\bf Homomorphisms of signed graphs:\\ An update}}
	
\date{\today} 
	
	\author[IRIF]{Reza Naserasr}
	\author[LaBRI]{\'Eric Sopena}
    \author[BING]{Thomas Zaslavsky}

	\address[IRIF]{Université de Paris, IRIF, CNRS, F-75013 Paris, France.  E-mail: reza@irif.fr}
	\address[LaBRI]{ Univ. Bordeaux, Bordeaux INP, CNRS, LaBRI, UMR5800, F-33400 Talence, France}
	\address[BING]{Department of Mathematical Sciences, Binghamton University, Binghamton, NY 13902-6000, U.S.A.}


\begin{abstract}

A signed graph is a graph together with an assignment of signs to the edges. A closed walk in a signed graph is said to be positive (negative) if it has an even (odd) number of negative edges, counting repetition. Recognizing the signs of closed walks as one of the key structural properties of a signed graph, we define a homomorphism of a signed graph $(G,\sigma)$ to a signed graph $(H, \pi)$ to be a mapping of vertices and edges of $G$ to (respectively) vertices and edges of $H$ which preserves incidence, adjacency and the signs of closed walks.

In this work we first give a characterization of the sets of closed walks in a graph $G$ that correspond to the set of negative walks in some signed graph on $G$.  We also give an easy algorithm for the corresponding decision problem.

After verifying the equivalence between this definition and earlier ones, we discuss the relation between homomorphisms of signed graphs and those of 2-edge-colored graphs. Next we provide some basic no-homomorphism lemmas. These lemmas lead to a general method of defining chromatic number which is discussed at length. Finally, we list a few problems that are the driving force behind the study of homomorphisms of signed graphs.

\end{abstract}


\end{frontmatter}

\newpage
\tableofcontents


\newpage
\newpage
\section{Introduction}
\label{sec:introduction}

The notion of homomorphisms of signed graphs was first defined by B. Guenin in an unpublished manuscript. The development of the subject started from \cite{NRS15}. It also appeared under a different context at \cite{BG09}. The theory extends the classical notion of graph homomorphism and is strongly related to the theory of homomorphisms of 2-edge-colored graphs, but has the main advantage of correlating with the theory of graph minors.

Being unsatisfied with the definition and terminology employed in \cite{NRS15}, here we present a more natural definition which leads to a number of simple no-homomorphism lemmas, special subclasses and some characterization theorems. The new definition is based on closed walks and their signs, thus we study them at length. The notion of chromatic number of a signed graph, introduced long ago by Zaslavsky \cite{Z84}, has recently drawn considerable attention with a variety of definitions offered. We will show how one such definition can be regarded as an optimization question with certain restrictions. 
At the end, we list only a few of the most motivating questions of this theory.

Toward a comprehensive study, we will begin with terminology and notation. This is of particular importance due to the diversity of terminology being used in the study of signed graphs. The following section is on the signs of cycles and closed walks, which we consider as basic elements of signed graphs. Section~\ref{sec:Homomorphism} is on the main subject of this work, homomorphisms of signed graphs, and the last section is about motivating problems, noting that the number of open problems is too large to be listed.

\section{Terminology and Notation for Graphs and Signed Graphs}
\label{sec:terminology}

\subsection{Graphs}\label{sec:graphterm}

We allow a graph $G=(V(G), E(G))$ to have loops and multiple edges.  Thus the vertex and edge sets are disjoint, and each edge is provided with a multiset of two vertices, called its \emph{endpoints}.  If the endpoints are equal, the edge is a \emph{loop}; if not, it is a \emph{link}.  Different edges may have the same multiset of endpoints; then they are called \emph{parallel edges}.  By having a set of edges instead of a multiset, we can have functions that differ on parallel edges.

An \emph{edge cut} $[X,Y]$ of $G$ is the set of all edges with one endpoint in $X$ and another in $Y$, where $X$ and $Y$ form a partition of $V(G)$. 
  
A \emph{walk} of $G$ is a sequence $W = v_0 e_1 v_1 e_2 \cdots, e_k v_k$  where for each $e_i$ its endpoints are $v_{i-1}$ and $v_{i}$ (thus if $e_i$ is a loop we must have $v_i=v_{i-1}$).  Its length is $k$ and its parity is the parity of $k$.  It is a \emph{closed walk} if we have $v_k=v_0$.  
An \emph{trivial walk} is a walk of length $0$, i.e., with one vertex and no edge; this is considered a closed walk.
The \emph{inverse} of $W$ is the walk $W\inv = v_ke_k \cdots e_2 v_1 e_1 v_0$.
The walk $W$ is a $\emph{path}$ if there is no repeated element. It is a $\emph{cycle}$ if $k\geq1$, $v_0=v_k$, and that is the only repetition of elements in the walk. Thus loops are cycles of length 1 and parallel edges form cycles of length 2.  Given two walks $W_1=v_0 e_1 v_1 e_2 \cdots e_k v_k$ and $W_2= v_k e_{k+1} v_{k+1} e_{k+2} \cdots e_{k+l} v_{k+l}$ ($W_2$ starts at the end  of $W_1$) the walk $v_0 e_1 v_1 e_2 \cdots e_k v_ke_{k+1} v_{k+1} e_{k+2} \cdots e_{k+l} v_{k+l}$ is denoted by $W_1W_2$. 
We often omit the vertices from the sequence defining a walk, since they are usually obvious; thus $W$ may be written as $e_1e_2\cdots e_k$.  
A closed walk $W = v_0 e_1 v_1 e_2 \cdots e_k v_k$ may be \emph{rotated} to begin at a different vertex, giving a walk $v_{i-1}e_iv_i \cdots e_kv_ke_1v_1 \cdots e_{i-1}v_i$, which we call a \emph{rotation} of $W$.

A \emph{theta graph} is a graph that consists of three paths joining the same two vertices, but which are otherwise pairwise disjoint.  Theta subgraphs of a graph are important in signed graph theory.

\subsection{$2$-edge-colored graphs}\label{sec:2ecgraphterm}

A \emph{$2$-edge-colored graph}, $\Gamma=(V(\Gamma), E_1(\Gamma), E_2(\Gamma))$, 
is a graph where the set of edges is decomposed into two disjoint subsets.  That is, we have two colors and we associate one of them to each edge of $\Gamma$.  Thus, we have a mapping $E(\Gamma) \to \mathbf{C}$, where $\mathbf{C}$ is the set of two colors.  

\subsection{Signed graphs}\label{sec:sgterm}

A \emph{signed graph} is a graph $G$ together with an assignment $\sigma: E \to \{+,-\}$ of a sign ($ +$ or $-$) to each edge of $G$.  We call $G$ the \emph{underlying graph}, and $\sigma$ is called the \emph{signature}.  We may denote this signed graph by $(G, \sigma)$ or sometimes by $\hG$. 
A signed graph where all edge are positive is denoted by $(G,+)$ and called \emph{all positive}, and similarly if all edges are negative it will be denoted by $(G,-)$ and called \emph{all negative}. 

We call a signed graph $(G,\sigma)$ connected, bipartite, etc.\ when $G$ is connected, bipartite, etc. 

The signed graph $(G,\sigma)$ may be thought of as a 2-edge-colored graph $(G,E^+,E^-)$, where $E^+$ and $E^-$ denote the sets of positive and negative edges, respectively; but that does not express the fact that signs $+$ and $-$ are essentially different.
The difference between a signed graph and a 2-edge-colored graph is  on the notion of \emph{sign of a closed walk} in $(G,\sigma)$.  For any walk $W = e_1e_2\cdots e_l$, of a signed graph $(G,\sigma)$, the sign of $W$ is $\sigma(W) := \sigma(e_1)\sigma(e_2)\cdots\sigma(e_l)$.  Then $W$ is said to be \emph{positive} or \emph{negative} depending on the value of $\sigma(W)$.  Since a cycle of a graph is also a closed walk we naturally have the definition of \emph{positive cycles} and \emph{negative cycles}.  It is clear that the signs of cycles determine the signs of all closed walks.

Signs of cycles determine many fundamental properties of a signed graph.  The most important is \emph{balance}. A signed graph $(G,\sigma)$ is said to be \emph{balanced} if every cycle is positive. It is said to be \emph{antibalanced} if $(G,-\sigma)$ is balanced.  
These two subclasses of signed graphs together with the subclass of signed bipartite graphs form three subclasses of signed graphs which will be shown to be of special interest in Section~\ref{sec:special}.

A closely related notion is the notion of \emph{switching}\footnote{Switching is called \emph{resigning} in some works on signed graphs.}: to switch a vertex $v$ of a signed graph $(G,\sigma)$ is to negate all signs in the edge cut $[\{v\}, V(G)\setminus v]$.  To switch a set $X$ of vertices is to switch all the vertices of $X$ in any sequence; it has the effect that it negates the edges in the cut $[X,Y]$, which is the same as switching the complementary set $Y$. 
A fundamental though obvious fact is that \emph{switching does not change the signs of closed walks}.   
An equivalent property to balance is that, after a suitable switching, every edge is positive; equivalent to antibalance is that, after suitable switching, every edge is negative.

Historically, a notion similar to that of edge signs in the form of a distinguished edge set appeared in the first book on graph theory, published by D. K\"onig in 1936 \cite[Chapter X, Section 3]{K36}.  This work contains many basic results about signed graphs (in its own terminology), as detailed in its entry in \cite{Z98}. The essential concepts of edge and cycle signs were first introduced by Harary \cite{H53}.  Switching in the form of set summation with a vertex cut is employed by K\"onig, but further use of switching of signed graphs begins only much later in \cite{Z81, Z82b}.

Two signatures $\sigma_1$ and $\sigma_2$ on the same graph $G$ are said to be \emph{switching equivalent} if one is obtained from the other by switching. The set of negative closed walks or negative cycles determines the class to which a signature belongs.

\begin{lemma}[Zaslavsky~{\cite[Theorem 3.2]{Z82b}}]\label{lem:Zaslavsky}
	Given two signatures $\sigma_1$ and $\sigma_2$ of a graph $G$, 
	$\sigma_1$ is a switching of $\sigma_2$ if and only if the sets of 
	positive (or, equivalently, negative) cycles of $(G,\sigma_1)$ and $(G,\sigma_2)$ are the same. 

\end{lemma}

The proof given in \cite{Z82b} is based on two key observations which we would like to state separately:  

\begin{lemma}\label{lem:Tree}
	For a tree $T$, any two signatures $\sigma_1$ and $\sigma_2$ are switching equivalent.
	\end{lemma}

 \begin{lemma}\label{lem:ForcingTree}
 	Given a connected graph $G$, a signature $\sigma$ of $G$, a spanning tree $T$ and signature $\tau$ of $T$ there is a unique signature $\sigma'$ which is identical to $\tau$ on the edges of $T$ and is switching equivalent to $\sigma$.
 \end{lemma}
 
 The proof of this lemma is based on the following key notions: given a spanning tree $T$ of $G$ and an edge $e$ of $G$ which is not in $T$, the subgraph induced by $e$ and edges of $T$ forms a unique cycle in $G$, denoted by $C_e$. The cycles $C_e$ are called the \emph{fundamental cycles} of $G$ with respect to $T$.  The set of fundamental cycles with respect to a spanning tree is called a \emph{fundamental system of cycles} for $G$. The lemma is then proved by observing the relation between the sign of $e$ and the sign of $C_e$.

The remark following \cite[Theorem 2]{Z81} states a stronger result: that there is a one-to-one correspondence between the classes of switching equivalent signatures on $G$ and subsets of a fundamental system of cycles of $G$, regarded as the positive fundamental cycles.  
It follows that the number of signatures on $G$ which are pairwise not switching equivalent is $2^{e-n+c}$ where $e$ is the number of edges, $n$ is the number of vertices and $c$ is the number of connected components. 
 
 These lemmas together also imply a fast algorithm to decide whether two signatures $\sigma_1$ and $\sigma_2$ on a graph $G$ are switching equivalent: Choose a spanning tree $T$. Switch both $\sigma_1$ and $\sigma_2$ to $\sigma'_1$ and $\sigma'_2$ so that they agree on $T$. Check if $\sigma'_1=\sigma'_2$.
This simple method, presented in \cite[Lemma 3.1]{Z82b}, is essentially the same as the test for balance found independently by Hansen~\cite{H78} and Harary and Kabell~\cite{HK80}, which switches $\sigma$ to be all positive on $T$ and deduces balance of $(G,\sigma)$ if and only if there is no negative edge in $(G,\sigma)$.
 
 \begin{proposition}\label{coro:SwitchClassPolytime}
 	It can be decided in time quadratic in the order of a graph $G$ whether two given signatures on $G$ are switching equivalent.
 \end{proposition}

\begin{proof}
Here is a 6-step algorithm, assuming $G$ is connected.  After Step 1, it is the algorithm of Hansen \cite{H78} and Harary--Kabell \cite{HK80}.

\begin{enumerate}[Step 1.]

\item Multiply the two signatures: $\sigma(e) = \sigma_1(e)\sigma_2(e).$  Time:  $O(n^2)$.  It is easy to see that $\sigma_1$ and $\sigma_2$ are equivalent if and only if $\sigma$ is balanced.

\item Find a rooted spanning tree, $(T,r)$.  Time:  $O(n^2)$.

\item Label each vertex $u$ by the sign in $(G,\sigma)$ of the path from $r$ to $v$ in $T$.  Time:  $O(n^2)$ at worst.  This step can be integrated into Step 2.

\item Find the set $X$ of negative vertices.  Time:  $O(n)$.

\item Negate the signs of edges in the cut $[X,V(G)\setminus X]$.  Time:  $O(n^2)$.  Note that now all edges in $T$ are positive.

\item Search for a negative edge.  Time:  $O(n^2)$.  

\end{enumerate}

If there is no negative edge, $\sigma$ is balanced and the original signatures are switching equivalent.  If there is a negative edge, $\sigma$ is unbalanced and the original signatures are not equivalent.
\end{proof}


\section{Positive and Negative Elements of the Cycle Space}\label{sec:CycleSpace}
Given a signed graph $\hG$, the set of cycles and the set of closed walks of the underlying graph $G$ are each divided into the two sets of positive and negative elements.  We think of the signature as determining sets of negative cycles and closed walks.

These subsets are among the fundamental properties of a signed graph, along with incidence and adjacency; thus we have two basic questions:
Given a subset $\cW$ of the closed walks in a graph $G$, does $\cW$ correspond to a set of negative (equivalently, positive) closed walks in $(G,\sigma)$ for some signature $\sigma$ of $G$? If yes, can we find one such $\sigma$ by an efficient algorithm?

One might emphasize the set of positive cycles or positive closed walks, especially because the former forms a subspace of the cycle space of $G$. However, in the study of signed graphs and specially in the study of homomorphism which is the central part of this work, negative cycles and closed walks are of higher importance. 
In fact, negative cycles extend the role which odd cycles play in many parts of graph theory.  
For this reason, some authors have used the term ``odd cycle" in a signed graph to refer to negative cycles, even when they have even length!

Observe that the set of all cycles of $G$ could be of exponential order (in the order of $G$) and that as long as $G$ contains an edge, the set of closed walks in $G$ is an infinite set. However, as we mentioned in the previous subsection and will discuss later in this section, the subset of negative closed walks can be finitely presented, e.g., by a subset of a fundamental system of cycles of $G$. 

Considering cycles rather than all closed walks, an analogue of the first question has already been addressed in \cite{Z81} using the following definition:

\begin{definition}[Theta Co-Additivity] 
		A set $\cC$ of cycles of a graph $G$ is \emph{theta co-additive} if, 
		for every theta subgraph $\Theta$ of $G$, the number of cycles in $\Theta$ 
		that belong to $\cC$ is even.
\end{definition}

It is then proved that:

\begin{lemma}[\cite{Z81}]\label{lem:Zaslavsky81}
	A set $\cC$ of cycles in a graph $G$ is the set of 
	negative cycles of $(G,\sigma)$ for some choice of signature $\sigma$ if and only if $\cC$ satisfies theta co-additivity.
\end{lemma}


\subsection{Systems of closed walks}\label{sec:signswalks}

As we will see later, in the study of homomorphisms of signed graphs we may have to consider closed walks that are not cycles; for example see Lemma~\ref{lem:NoHomGirth} and Figure~\ref{fig:girth-by-closed walks}. Thus we introduce the following notion which, working with the set of closed walks rather than just cycles, extends the notion of theta co-additivity. 
Then we provide a similar tool to test whether a given set $\cW$ of closed walks in $G$ is the exact set of negative closed walks  of $(G, \sigma)$ for some choice of a signature $\sigma$. This test can, in particular, be used for a NO-certificate. We then present an algorithm by which we can answer the algorithmic part of the question and produce one such signature if the output is YES. Our algorithm is efficient with the condition that $\cC$ is presented efficiently and that testing membership in $\cC$ can be done efficiently. Our first test is based on the following definitions.

\begin{definition}[Rotation property]\label{def:rotation} A set $\cW$ of closed walks is said to have the \emph{rotation property} if, for each closed walk $PQ$ (the concatenation of walks $P$ and $Q$), 
either $PQ$ and $QP$ are both in $\cW$ or neither of them is in $\cW$.
\end{definition}

\begin{definition}[Exclusive 3-walk property]\label{def:ex3walk} 
		A set $\cW$ of closed walks in a graph $G$ satisfies the \emph{exclusive 3-walk property} if 
		it has the rotation property and, 
		for every two vertices $x$ and $y$ (not necessarily distinct) of $G$ and
		every three $xy$-walks $W_1$, $W_2$ and $W_3$ of $G$, 
	$$W_1W_2\inv \in \cW \Rightarrow (W_1W_3\inv \in \cW) \veebar (W_2W_3\inv \in \cW),$$
		where $\veebar$ denotes the ``exclusive or'' operator.
		That is to say, given any two vertices $x$ and~$y$, among the three closed walks induced by 
		any three $xy$-walks, an even number is in the set $\cW$, i.e., either none of them or exactly two of them. 
\end{definition}

(We thank Andrzej Szepietowski \cite{Sz} and an anonymous referee for pointing out the need to assume the rotation property.)

We give algebraic formulations.  The \emph{characteristic sign function} of closed walks with respect to $\cW$ is
 $$	\sigma_\cW(W) = \begin{cases}
	+  &\text{if }W \notin \cW, \text{ and}\\ 
	- &\text{if } W \in \cW. \end{cases}
	$$
In terms of this function, 
Definition \ref{def:rotation} becomes the formula
\begin{equation}
	\sigma_\cW(PQ) = \sigma_\cW(QP)
\label{eq:rotation}
\end{equation}
if $P$ is an $xy$-walk and $Q$ is a $yx$-walk,
and the implication of Definition \ref{def:ex3walk} becomes the formula
\begin{equation}
	\sigma_\cW(W_1W_2\inv) \sigma_\cW(W_1W_3\inv) \sigma_\cW(W_2W_3\inv)  = +.
\label{eq:ex3walk}
\end{equation}

Using this definition, we can characterize the sets $\cW$ of closed walks in a graph $G$ that can be the set of negative closed walks in $(G,\sigma)$ for some choice of signature $\sigma$, by the fact that they satisfy the exclusive 3-walk property (see Theorem~\ref{thm:3-walks}). 
As it is easy to prove that the set of negative closed walks of a signed graph satisfies the property, our main goal is to show that if a set $\cW$ of closed walks in $G$ satisfies the exclusive 3-walk property then $\cW$ is the set of negative closed walks in $(G,\sigma)$ for a signature $\sigma$ of $G$. Being such a set would imply certain properties, for example that a trivial walk cannot be a member of $\cW$ as it is positive by definition. We collect such properties in the following proposition and provide a proof solely based on the exclusive 3-walk property.


\begin{proposition}\label{prop:Exclusive3WalkBasicProperties}

Let $G$ be a graph and let $\cW$ be a set of closed walks which satisfies the exclusive 3-walk property. Then $\cW$ satisfies the following properties:

\begin{itemize}
	\item[{\rm[i]}] No trivial walk is in $\cW$. 
		(I.e., $\sigma_\cW(W_0) = +$ for every trivial walk $W_0$.)
	\item[{\rm[ii]}] For any walk $W$, the closed walk $WW\inv$ is not in $\cW$. 
		(I.e., $\sigma_\cW(WW\inv)=+$.)
	\item[{\rm[iii]}] For any closed walk $W$, we have $W\inv \in \cW$ if and only if $W\in \cW$.
		(I.e., $\sigma_\cW(W\inv) = \sigma_\cW(W)$.)
	\item[{\rm[iv]}] For any pair of closed walks $W$ and $W'$ with the same end point $v$, $WW'$ is in $\cW$ if and only if exactly one of $W$ and $W'$ is in $\cW$. 
	(I.e., $\sigma_\cW(WW') = \sigma_\cW(W) \sigma_\cW(W')$.)

	\item[{\rm[v]}] Given a closed walk $W$ starting at $y$ and an $xy$-walk $P$, $PWP\inv$ is in $\cW$ if and only if $W$ is in $\cW$.
		(I.e., $\sigma_\cW(PWP\inv) = \sigma_\cW(W)$).


\end{itemize}	
\end{proposition}


\begin{proof}
	By taking $W$ as a trivial walk we observe that the first claim is a special case of the second. Thus we jump to proving the second claim. For this, we take $W_1=W_2=W_3=W$. Then the three closed walks in the exclusive 3-walk property are all $WW\inv$ but a condition for $\cW$ is that all three cannot be in $\cW$. 

	For [iii], let $v$ be the start of $W$ and for $x=y=v$ take $W_1=W$ and $W_2=W\inv$ and let $W_3$ be the trivial walk at $v$. Then the three $xy$-walks to be considered are $W$, $W\inv$ and $WW\inv$. By part [ii], the last one is never in $\cW$, thus we have our claim.  
	
	For [iv] take $x=y=v$ and consider $W_1=W$, $W_2=v$ (the trivial walk at $v$) and $W_3=(W')\inv$. Then the three walks considered in the exclusive 3-walk property are $W$, $W'$ and $WW'$. 
	By Equation \ref{eq:ex3walk}, $\sigma_\cW(W) \sigma_\cW(W') \sigma_\cW(WW') = +.$

	For [v] rotate $PWP\inv$ to $WP\inv P$; then apply [iv] and [ii].
%
\end{proof}

\begin{theorem}\label{thm:3-walks}
	A set $\cW$ of closed walks in a graph $G$ is the set of negative closed walks in $(G,\sigma)$ 
	for some choice of signature $\sigma$ if and only if $\cW$ satisfies the exclusive 3-walk property.
\end{theorem}

\begin{proof}
We will prove the theorem for connected graphs. For graphs with more than one connected component we may apply the proof to each connected component of the graph.

First we consider a signed graph $(G,\sigma)$ and we show that the set $\cW^{-}$ of the negative closed walks in $(G, \sigma)$ satisfies the exclusive 3-walk property. Consider two vertices $x$ and $y$ (not necessarily distinct) and let $W_1$, $W_2$ and $W_3$ be three $xy$-walks in $(G, \sigma)$. Then 
$$\sigma(W_1W_2\inv) \sigma(W_1W_3\inv) \sigma (W_2W_3\inv) = \sigma(W_1)^2 \sigma (W_2)^2 \sigma(W_3)^2 = +,$$
thus verifying the exclusive 3-walk property in the form of Equation \eqref{eq:ex3walk}.	
	
	To prove the converse, let $\cW$ be a set of closed walks in $G$ satisfying the exclusive 3-walk property. We need to show that there exists a signature $\sigma$ of $G$ such that the set of negative closed walks in $(G, \sigma)$ is exactly $\cW$.  
A cycle $C$ of $G$ can be viewed as a closed walk starting at a vertex of the cycle. By Proposition~\ref{prop:Exclusive3WalkBasicProperties} [vi], whether this closed walk is in $\cW$ or not is independent of the choice of the starting vertex. Thus we may simply talk about a cycle being in $\cW$. We then use a fundamental system of cycles of $G$ to define a signature $\sigma$. 
	Let $T$ be a spanning tree of $G$ and for each edge $e \notin T$ let $C_e$ be the fundamental cycle, with respect to $T$, that contains $e$.  Define a signature $\sigma$ by
$$	\sigma(e) = \begin{cases}
	+	&\text{if } e \in T, \\
	\sigma_\cW(C_e)	&\text{if } e \notin T.
	\end{cases}
	$$
	
	It remains to show that $\sigma(W) = \sigma_\cW(W)$ for every closed walk $W$.  We prove this by induction on the number of edges that are not in $T$.  By Proposition~\ref{prop:Exclusive3WalkBasicProperties} a trivial walk $W_0$ is not in $\cW$ so $\sigma_\cW(W_0) = + = \sigma(W_0)$.  Let $W$ be a closed walk which uses only edges of $T$.  As $T$ has no cycle, the induced subgraph by the edges of $W$ must have a vertex $y$ of degree 1. Let $x$ be the neighbour of this vertex in $T$. Let $P$ be the walk $xey$ where $e=xy$. Then $W$ can be written as $W'PP\inv$, where $W'$ has two fewer edges than $W$. By Proposition~\ref{prop:Exclusive3WalkBasicProperties} 
	$\sigma_\cW(W') = \sigma_\cW(W)$. By repeating this process we conclude that $\sigma_\cW(W) = \sigma_\cW(W_0) = + = \sigma(W)$.

	To complete the induction we consider a closed walk $W$ which uses $k$ edges not in $T$ (counting repetitions), $k\geq 1$. Let $e=xy$ be one such edge. We build a new walk $W'$ which uses $k-1$ edges not in $T$ and such that $\sigma_\cW(W') = \sigma_\cW(W) \sigma(e)$.  By Proposition~\ref{prop:Exclusive3WalkBasicProperties} we may rotate and, if necessary, invert $W$ so that it starts with $Q=xey$ and continues with $W''$, so $W = QW''$.  
	Let $P$ be the $xy$-path obtained from $C_e$ after removing the edge $e$. Observe that all edges of $P$ are in $T$. Let $W' = PW''$.  Note that $W'$ has one less edge not in $T$ than $W$ does.  

Consider the following three $xy$-walks: $W_1 = Q$, $W_2=P$ and $W_3 = (W'')\inv$. Then $W_1W_2\inv = C_e$, $W_1W_3\inv = W$ and $W_2W_3\inv = W'$.  By Equation \eqref{eq:ex3walk}, 
$$\sigma_\cW(W_1W_2\inv) \sigma_\cW(W_1W_3\inv) \sigma_\cW(W_2W_3\inv) = +;$$
more simply, 
$\sigma_\cW(C_e) \sigma_\cW(W) \sigma_\cW(W') = +.$
So, 
\begin{align*}
\sigma_\cW(W) &= \sigma_\cW(Ce) \sigma_\cW(W') = \sigma(e) \sigma_\cW(W') 
&\text{by the definition of $\sigma$,} \\
&= \sigma(e) \sigma(W') 
&\text{by induction,}\\
&= \sigma(e) \sigma(P) \sigma(W'') = \sigma(e) \sigma(W'') 
&\text{because all edges of $P$ are in $T$,}\\
&= \sigma(W).
&\qedhere
\end{align*}
\end{proof}

\subsection{Algorithmics of closed walk systems}\label{sec:algwalks}

Given a closed walk $W$ of graph $G$, let $C_{W}$ be the subgraph of $G$ which is induced by edges which appear on $W$ an odd number of times. 
Then $C_{W}$ has no repeated edges, so it is an edge-disjoint union of cycles.  Equivalently, every vertex has even degree in $C_W$; such a subgraph is called an \emph{even subgraph}. Given a signature $\sigma$ on $G$, the product of the signs of cycles in $C_{W}$ determines whether $W$ is a negative closed walk or a positive one.
Observe that while the set of closed walks in a graph with at least one edge is an infinite set, the set of even subgraphs is a finite set because each even subgraph uses each edge at most once. 
Thus in practice we may use a chosen set of even subgraphs (a finite set) to define a choice of negative (or positive) closed walks (an infinite set).  The question is then, what conditions on a set of even subgraphs make it the set of negative (or positive) even subgraphs $C_W$ of some signature on $G$.

Consider a subset $\cW$ of closed walks in $G$ that may, or may not, be the set of all negative closed walks in some signed graph on $G$.  Then the choice of $\cW$ corresponds to a choice of $C_{W}$'s. Thus a first necessary condition for $\cW$ to be the set of negative closed walks in $(G,\sigma)$ for some choice of $\sigma$ is that the two sets $\{C_{W}\mid W\in \cW\}$ and $\{C_{W} \mid W\notin \cW\}$ have no common element. 

Observe that the set of all even subgraphs corresponds to the (binary) cycle space of the graph $G$, which is the set of binary vectors in $\mathbb{Z}_2^{E}$, of length $|E(G)|$, which are in the null space of the vertex-edge incidence matrix of $G$ regarded as a binary matrix.  The (binary) characteristic vector of a closed walk $W$ equals that of the corresponding even subgraph $C_W$.  
If we identify a subset of edges with its characteristic vector, then binary addition corresponds to set summation (that is, symmetric difference).  
Assume $G$ is connected and let $W_1$ and $W_2$ be two closed walks in $G$. By adding a $PP\inv$ walk and rotating the walks, as needed, we may assume $W_1$ and $W_2$ are starting from the same vertex, observing that the $C_{W_i}$ will remain the same.  Then the set sum $C_{W_1}+C_{W_2}$, which corresponds to the binary sum of the characteristic vectors of $W_1$ and $W_2$, equals $C_{W_1W_2}$ (noting that $W_1W_2$ is only defined if they have a common starting point).

For a signed graph $(G, \sigma)$ the even subgraphs whose elements correspond to positive closed walks in $(G, \sigma)$ form a subspace of the cycle space of codimension at most 1, and every such subspace is the set of even subgraphs that are positive in some signature of $G$ \cite{Z81}. We may then use the notion of basis in linear algebra to present a choice of negative even subgraphs using only a linear number of edges, by using the notion of a fundamental system of cycles with respect to a spanning tree.  Let $T$ be a spanning tree of $G$.  For an edge $e \notin E(T)$, the \emph{fundamental cycle} of $e$ with respect to $T$ is the unique cycle contained in $T \cup e$.  An elementary consequence of Lemma \ref{lem:Zaslavsky} is that, given a signature $\sigma$ on $G$ with negative edge set $E^-$, defining $\sigma_T(e) := \sigma(C_e)$ if $e \notin E(T)$ and $\sigma_T(e) := +$ if $e \in E(T)$, then $\sigma_T$ is switching equivalent to $\sigma$.

Using this terminology, we state in the next theorem how a subset of all the closed walks (an infinite set) satisfying the exclusive 3-path property can be presented by a choice of a negative edge set. 

\begin{theorem}\label{thm:oddedges}
	Let $G$ be a connected graph with a spanning tree $T$. Let $\cW$ be a subset of the closed walks in $G$ that satisfies the exclusive 3-path property. Let $E^{-}$ be the set of edges $e\notin E(T)$, such that $C_e$ is in $\cW$. Then a closed walk $W$ is in $\cW$ if and only if $C_W$ can be written as the sum of $C_e$'s where an odd number of $e$'s in is $E^{-}$.  
\end{theorem}

\begin{proof}
By Theorem \ref{thm:3-walks}, the assumption on $\cW$ is equivalent to assuming that it is the class of negative closed walks of a signature $\sigma$ on $G$.  

If that is true, by Lemma \ref{lem:ForcingTree} we may switch $\sigma$ so that every edge of $T$ is positive.  Then $\sigma(W) = \sigma(C_W) = \prod_{e \in E(C_W) \setminus E(T)} \sigma(e).$  
\end{proof} 


\section{Homomorphisms}\label{sec:Homomorphism}

\subsection{Graphs}
The main goal of this work is the study of homomorphisms of signed graphs with special focus on improving terminology. To this end we first offer an alternative definition of a graph homomorphism as follows (recall that we are allowing loops and multiple edges): 

\begin{definition} 
A \emph{homomorphism} of a graph $G$ to a graph $H$ is a mapping $f$ which maps vertices of $G$ to vertices $H$ and edges of $G$ to the edges of $H$ and preserves the incidence relation between vertices and edges.  We indicate the existence of a homomorphism by writing $G \to H$.
\end{definition}

This contrasts with the usual definition, in which graphs are assumed to be simple, the vertex mapping is assumed to preserve adjacency (but not non-adjacency) and then the edge mapping is determined by the vertex mapping.  Since signed graphs have multiple edges that may have different signs, that kind of definition is inadequate.


\subsection{Signed graphs}

We may now define the central concept of this work:

\begin{definition} 
A \emph{homomorphism} of a signed graph $(G,\sigma)$ to a signed graph $(H, \pi)$, written $(G,\sigma) \to (H, \pi)$, is a graph homomorphism that preserves the signs of closed walks.   
	More precisely, this is a \emph{switching homomorphism} of signed graphs.
	We indicate the existence of a homomorphism by the notation $(G,\sigma) \to (H, \pi)$.

\end{definition}

Intuitively speaking, since we consider the signs of cycles of a signed graph as one of its determining characteristics, we define a homomorphism of signed graphs to be a mapping that preserves the main structures: incidences and cycle signs.

One can redefine homomorphisms of signed graphs in terms of edge signs.  A \emph{edge-sign-preserving homomorphism} of a signed graph $(G,\sigma)$ to a signed graph $(H, \pi)$ is a homomorphism of underlying graphs, $G \to H$, that preserves the signs of edges.

\begin{theorem}\label{thm:HomSecondDefinition}
	
	A homomorphism $(G,\sigma) \to (H,\pi)$ of signed graphs consists of a switching of $(G,\sigma)$ followed by an edge-sign-preserving homomorphism to $(H,\pi)$, and conversely.
\end{theorem} 

\begin{proof}
 We first show the easy part.  Suppose $(G,\sigma)$ switches to $(G,\sigma')$ such that there is a graph homomorphism $\phi:  G \to H$ that preserves edge signs.  Since $\phi$ preserves edge signs, it preserves the signs of all walks.  The signs of closed walks are the same in $(G,\sigma)$ and in $(G, \sigma')$, so $\phi$, considered as a mapping of signed graphs $(G,\sigma) \to (H,\pi)$, preserves the signs of closed walks.

	For the contrary, suppose $\phi$ is a signed graph homomorphism $(G,\sigma) \to (H, \pi)$; thus, $\phi$ is a graph homomorphism $G \to H$ and preserves the signs of closed walks.  Let $\sigma'$ be a new signature which assigns to each edge $e$ of $G$ the sign of $\phi(e)$ in $(H, \pi)$.  Then $\phi$, as a mapping from $(G,\sigma') \to (H,\pi)$, preserves the signs of closed walks and in particular the signs of cycles.  Each cycle $C$ in $G$ has the same sign with respect to $\sigma$ and $\sigma'$, as both signs equal the sign of $\phi(C)$ in $(H, \pi)$.  It then follows from Lemma~\ref{lem:Zaslavsky} that $\sigma'$ is a switching of $\sigma$. 
\end{proof}

The definition of a homomorphism of one signed graph to another, contained in this theorem, is the one originally given in~\cite{NRS15}, and in practice is easier to use.  To be more precise, in \cite{NRS15} a homomorphism of $(G,\sigma)$ to $(H, \pi)$ is a mapping $f=(f_1,f_2, f_3)$ where $f_1: V(G) \to \{+,-\}$ specifies  for each vertex $x$ whether $x$ is switched or not, $f_2: V(G) \to V(H)$ specifies to which vertex of $V(H)$ the vertex $x$ is mapped to and, similarly, $f_3:E(G)\to E(H)$ is the edge mapping. When working on graphs with no parallel edges, one may simply write $(f_1, f_2)$ as $f_3$ would be uniquely determined by $f_2$.

\subsubsection{Bipartite universality}

What is rather surprising, and not so obvious, is that the restriction to signed bipartite graphs also captures the classic notion of homomorphism of graphs.
With any graph $G$, we associate a signed bipartite graph $S(G)$ defined as follows: 
for each edge $uv$ of $G$, we first add a parallel edge, and then 
subdivide both edges in order to form a 4-cycle
(if $G$ has $n$ vertices and $m$ edges, $S(G)$ has thus $n+2m$ vertices and $4m$ edges). 
Finally, for each such 4-cycle, we assign one negative and three positive signs to its edges.  
That is, we replace each edge $uv$ of $G$ by a 4-cycle with one negative edge in which $u$ and $v$ are not adjacent. 
The following is proved in~\cite{NRS15}.

\begin{theorem}\label{thm:graph--signedbipartitegraph}
	Given graphs $G$ and $H$, there is a homomorphism of $G$ to $H$ if and only if there is a homomorphism of $S(G)$ to $S(H)$.
\end{theorem}

It is thus of special interest to study the homomorphism relation on the subclass of signed bipartite graphs.  Indeed, as is discussed in Section~\ref{sec:Future Work}, the study of homomorphisms of signed graphs was originated in order to address a possible extension of the Four-Color Theorem that is partly on the family of signed bipartite graphs and partly on that of antibalanced signed graphs.

\subsection{Core}

A common notion in the theory of homomorphisms is the notion of \emph{core}, which is defined analogously for each of the structural objects. 
A core is a graph (or a signed graph; or a 2-edge-colored graph---see Section \ref{Sec:Connection}) 
which does not admit any homomorphism to one of its proper subgraphs.
The \emph{core of a graph $G$} is then the smallest subgraph of $G$ (with respect to subgraph inclusion)
to which $G$ admits a homomorphism. It is clear, in each case, that the core of a graph is a core,
and it is not difficult to show that the core of a graph is unique up to isomorphism.
\subsection{Isomorphism and transitivity}

Any adaptation of a notion of homomorphism leads to a corresponding notion of isomorphism. 
An \emph{isomorphism}, more precisely a \emph{switching isomorphism}, from $(G,\sigma)$ to $(H,\pi)$ is a homomorphism $\phi$ that has an inverse, i.e., a homomorphism $\theta: (H,\pi) \to (G,\sigma)$ such that the compositions $\phi\circ\theta$ and $\theta\circ\phi$ are the identity self-mappings: $\phi\circ\theta = \id_{(H,\pi)}$ and $\theta\circ\phi = \id_{(G,\sigma)}.$  Equivalently, $\phi$ is a graph isomorphism that preserve the signs of closed walks; it follows easily that $\theta$ is also such a graph isomorphism.  We write $\phi\inv$ for $\theta$ and call it the \emph{inverse isomorphism} to $\phi$.
  
We say that $(G,\sigma)$ is \emph{isomorphic} to $(H,\pi)$ if there is an isomorphism of $(G,\sigma)$ to $(H,\pi)$

An \emph{automorphism} of a signed graph $(G,\sigma)$ is an isomorphism of $(G,\sigma)$ to itself. The set of all automorphisms of $(G,\sigma)$ forms the \emph{automorphism group} of $(G,\sigma)$.

We say, as with unsigned graphs, that a signed graph $(G,\sigma)$ is \emph{vertex transitive} if for any two vertices $x$ and $y$ of $G$,
there exists a automorphism of $(G,\sigma)$ which maps $x$ to $y$.  That is, the automorphism group is transitive on $V$.
We define \emph{edge transitivity} similarly.

\subsection{Girth and no-homomorphism lemmas}\label{sec:Question of bounding sparse graphs}

That the parity of a closed walk must be preserved by a homomorphism is the key fact that makes most graph homomorphism problems NP-hard. However, it also leads to a no-homomorphism lemma based on the easily computable parameter of odd girth (the least length of an odd cycle): the image of every odd cycle must contain an odd cycle which, therefore, is of smaller or equal length. Thus if $G\to H$, then the shortest odd cycle of $G$ is at least as big as the shortest odd cycle of $H$. 

With homomorphisms of signed graphs asking to preserve signs of closed walks, we have four types of  essentially different closed walk to consider: positive even closed walks, positive odd closed walks, negative even closed walks and negative odd closed walks. Given a homomorphism of $(G,\sigma)$ to $(H, \pi)$, images of closed walks of a given type in $(G,\sigma)$ must be of the same type in $(H, \pi)$.

We will say that a positive even walk is of type~$00$, a negative even walk is of type~$10$,
a positive odd walk is of type~$01$ and a negative odd walk is of type~$11$.
This notation is convenient in the following sense: given closed walks $W_1$  of type $ij$ and $W_2$ of type $i'j'$, both starting at a same vertex, the type of closed walk $W_1W_2$ is the binary sum of the two types, that is $ij+i'j'$ where addition is done in $\mathbb{Z}_2^2$.

We may now introduce four separate notions of girth based on walks, one for each of the four types of closed walk: 

\begin{definition}[Walk-girths of signed graphs]\label{def:girth}
Given a signed graph $(G, \sigma)$, for each walk type $ij$ we define the length of a shortest nontrivial closed walk of that type in $(G, \sigma)$ to be the \emph{$ij$-walk-girth} of $(G, \sigma)$ and denote it by $g_{ij}(G,\sigma)$. When there is no such walk, we write $g_{ij}(G,\sigma)=\infty$. 

For simplicity, we use the term \emph{girth of type $ij$} to denote the walk-girth of type $ij$. 
\end{definition}

As the examples of Figure~\ref{fig:girth-by-closed walks} show, a walk-girth of type $ij$ is not always obtained by a cycle; thus we use term walk-girth. 

\begin{figure}[hbt]
	\center
	\begin{tikzpicture}[scale=0.6]
	\draw (-2,0)  node[circle,fill=black,inner sep=1mm,minimum size=3mm] (v1) {};
	\draw (-.5,3)  node[circle,fill=black,inner sep=1mm,minimum size=3mm] (v2) {};
	\draw (-.5,-3)  node[circle,fill=black,inner sep=1mm,minimum size=3mm] (v3) {};
	\draw (-3.5,3)  node[circle,fill=black,inner sep=1mm,minimum size=3mm] (v4) {};
	\draw (-3.5,-3)  node[circle,fill=black,inner sep=1mm,minimum size=3mm] (v5) {};

	\draw [line width=0.6mm, blue ] (v1) -- (v2);
	\draw [line width=0.6mm, blue ](v1) -- (v3);
	\draw [line width=0.6mm, blue ](v1) -- (v4);
	\draw [line width=0.6mm, blue ](v1) -- (v5);
	\draw [line width=0.6mm, blue ](v4) -- (v2);
	\draw[line width= 0.6mm,dashed, red ] (v3) -- (v5);
	
	\draw (5,0)  node[circle,fill=black,inner sep=1mm,minimum size=3mm] (u1) {};
	\draw (6.5,3)  node[circle,fill=black,inner sep=1mm,minimum size=3mm] (u2) {};
	\draw (6.5,-2.5)  node[circle,fill=black,inner sep=1mm,minimum size=3mm] (u3) {};
	\draw (3.5,3)  node[circle,fill=black,inner sep=1mm,minimum size=3mm] (u4) {};
	\draw (3.5,-2.5)  node[circle,fill=black,inner sep=1mm,minimum size=3mm] (u5) {};
	\draw (5,-5)  node[circle,fill=black,inner sep=1mm,minimum size=3mm] (u6) {};
	
	\draw [line width=0.6mm, blue ](u1) -- (u2);
	\draw [line width=0.6mm, blue ](u1) -- (u3);
	\draw [line width=0.6mm, blue ](u1) -- (u4);
	\draw [line width=0.6mm, blue ](u1) -- (u5);
	\draw [line width=0.6mm, blue ](u4) -- (u2);
	\draw [line width=0.6mm, blue ](u5) -- (u6);
	\draw[line width= 0.6mm,dashed, red] (u3) -- (u6);
	
	\draw (12,0)  node[circle,fill=black,inner sep=1mm,minimum size=3mm] (w1) {};
	\draw (13.5,3)  node[circle,fill=black,inner sep=1mm,minimum size=3mm] (w2) {};
	\draw (13.5,-2.5)  node[circle,fill=black,inner sep=1mm,minimum size=3mm] (w3) {};
	\draw (10.5,3)  node[circle,fill=black,inner sep=1mm,minimum size=3mm] (w4) {};
	\draw (10.5,-2.5)  node[circle,fill=black,inner sep=1mm,minimum size=3mm] (w5) {};
	\draw (12,-5)  node[circle,fill=black,inner sep=1mm,minimum size=3mm] (w6) {};
	
	\draw [line width=0.6mm, blue ](w1) -- (w2);
	\draw [line width=0.6mm, blue ](w1) -- (w3);
	\draw [line width=0.6mm, blue ](w1) -- (w4);
	\draw [line width=0.6mm, blue ](w1) -- (w5);
	\draw[line width= 0.6mm,dashed, red] (w4) -- (w2);
	\draw[line width=0.6mm, blue ](w5) -- (w6);
	\draw[line width= 0.6mm,dashed, red] (w3) -- (w6);
	\end{tikzpicture}

	\caption{Signed graphs where $g_{10}$, $g_{11}$, $g_{01}$ are realized by a closed walk but not a cycle.} 
	\label{fig:girth-by-closed walks}
\end{figure}
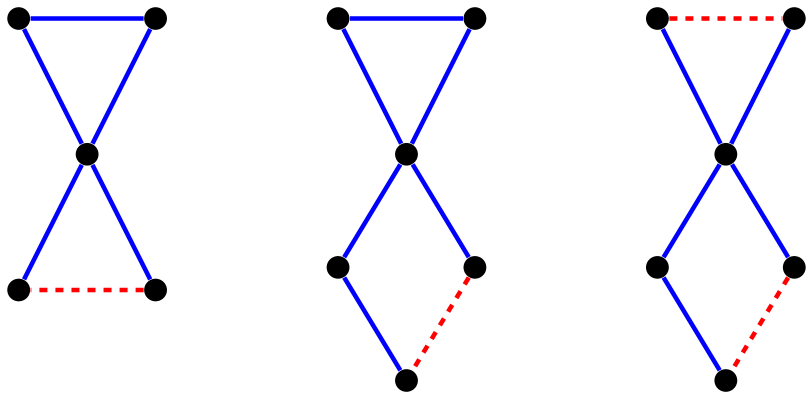 

Observe that $g_{00}(G,\sigma)$ is either $\infty$ (when $G$ has no edge) or it is 2. 
While $g_{00}$ is thus not of much interest, the other three values are of high interest in the study of homomorphisms of signed graphs. Indeed we have the following no-homomorphism lemma from the definition.
 
 \begin{lemma}\label{lem:NoHomGirth}
 	If a signed graph $(G,\sigma)$ admits a homomorphism to a signed graph $(H,\pi)$, then $$g_{ij}(G,\sigma)\geq g_{ij}(H,\pi)$$ for each $ij\in \mathbb{Z}_2^2$.
 \end{lemma}

To emphasize the importance of Lemma~\ref{lem:NoHomGirth}, we show that given a signed graph $(G, \sigma)$ one can compute $g_{ij}(G,\sigma)$ for all four choices of $ij$ in polynomial time.

\begin{proposition}\label{prop:Computing g_ij}
	Given a signed graph $(G, \sigma)$ of order $n$ we can compute $g_{ij}(G,\sigma)$ in time $O(n^4)$.
\end{proposition}

\begin{proof}
	Our method is that, for each vertex $v$ of $G$, we compute in time $O(n^3)$ the length of a shortest closed walk of type $ij$ which starts at $v$. Taking the minimum of all such values then gives $g_{ij} (G,\sigma)$. (The length of the input is $O(n^2)$ because we assume no parallel edges in $(G,\sigma)$ have the same sign.)

	We use the following notation: $N_i(v)$ is the set of vertices that are at distance $i$ from $v$ in $G$.
Moreover, for $\epsilon = +, -$, we define $N_i^\epsilon(v)$ to be the set of vertices at distance $i$ from $v$ that can be reached from $v$ by a path $P$ of length $i$ and sign $\epsilon$. Observe that these sets can be built inductively: given $N_i^+(v)$ and $N_i^-(v)$ for $i = k-1$, vertices not already reached which are adjacent to some vertex in $N_i^+(v)$ with an edge of sign $\epsilon$ or to some vertex in $N_i^-(v)$ with an edge of sign $-\epsilon$ form $N^\epsilon_{k}(v)$.

The time to construct the sets $N_k^\epsilon(v)$ for each vertex $v$ is $O(n^3)$.  Each set satisfies $|N_k^\epsilon(v)| = O(n)$.  Each vertex in $N_{k-1}^\epsilon(v)$ scans $O(n)$ vertices for possible inclusion in $N_k^\epsilon(v)$.  This procedure is repeated $O(n)$ times for increasing values of $k$.  The time needed for this is $O(n^3)$ for each vertex $v$; thus, $O(n^4)$ in total.

	The shortest length of a closed walk of type $ij$ starting at $v$, denoted by $g_{ij}(G,v,\sigma)$, is now computed as follows.

	\begin{enumerate}[$ij=00$]
	\item [$ij=00$:] $g_{00}(G, v,\sigma)=2$ if $N_1(v) \neq \emptyset$ and $g_{00}(G,v, \sigma)=\infty$ otherwise.
		
	\item [$ij=01$:] Consider the first $k$ where one of the following happens:  $N^+_{k}(v)$ or $N^-_{k}(v)$ induces a positive edge, or a negative edge connects a vertex from $N^+_{k}(v)$ to a vertex from $N^-_{k}(v)$. Then  $g_{01}(G,v,\sigma)=2k+1$.
	    
	\item [$ij=10$:] Consider the first $k$ where $N^+_{k}(v) \cap N^-_k(v)\neq \emptyset$.  Then  $g_{10}(G,v, \sigma)=2k$.
	    
	\item [$ij=11$:] Consider the first $k$ where one of the following happens:  $N^+_{k}(v)$ or $N^-_{k}(v)$ induces a negative edge, or a positive edge connects a vertex from $N^+_{k}(v)$ to a vertex from $N^-_{k}(v)$. Then  $g_{11}(G,v,\sigma)=2k+1$.

	\end{enumerate}
	
This computation takes time $O(n^3)$ because there are $O(n)$ possible values of $k$ and the set intersections and edge examinations take time $O(n^2)$.

	Now the $n$ vertices are examined to find $g_{ij}(G,\sigma) = \min_v g_{ij}(G,v,\sigma)$.  This takes time $O(n)$.

	The longest step is the first, so the whole procedure takes time $O(n^4)$.
\end{proof}

The bound of $O(n^4)$ is for simplicity of the proof; one can certainly do better. We conjecture that this can be done in time $O(n^3)$. 

Observe that if $g_{00}(G,\sigma)=\infty$ then $G$ has no edge and, therefore, $g_{ij}(G,\sigma)=\infty$ for all other choices of $ij$. In the next lemma we show something similar but weaker for other choices of $ij$: that if $G$ is connected and $g_{ij}(G,\sigma)=\infty$, then at least for one other choice $i'j'$ we have $g_{i'j'}(G,\sigma)=\infty$. This will lead to the study of three special subclasses of signed graph, each of which is of special importance in the study of homomorphisms. These classes will be discussed in the next subsection.

\begin{lemma}\label{lem:2infty}
Let $(G,\sigma)$ be a connected signed graph.
If $(G,\sigma)$ contains two closed walks $W_1$ and $W_2$ of types $i_1j_1\neq 00$ and $i_2j_2\neq 00$, respectively,
with $i_1j_1\neq i_2j_2$,
then it contains a closed walk $W_3$ of the third nonzero type.

\end{lemma} 

\begin{proof}
Since $G$ is connected, there is a (shortest) path $P$ connecting a vertex $u$ of $W_1$ to a vertex $v$ of $W_2$. 
Let $W_3=W_1PW_2P\inv$, as a closed walk starting at $u$. 
Since $P$ is traversed twice, it affects neither parity nor sign of $W_3$. Therefore, $W_3$ is of type $(i_1+j_1)(i_2+j_2)$ (where addition is modulo $2$). Since $i_1j_1$ and $i_2j_2$ are nonzero elements of $\mathbb{Z}_2^2$, their sum is the third nonzero element of this group. 
\end{proof}

We say that $g_{ij}(G,\sigma)$ is \emph{realized} by a walk $W$ if $W$ has the type $ij$ and has minimum length with that type.  It appears that $g_{ij}(G,\sigma)$, if finite, may be realized by cycles and non-cyclic closed walks, or only by cycles, or only by closed walks that are not cycles.  Appearances are misleading, as we see in the following result, where for simplicity we write $g_{ij} := g_{ij}(G,\sigma)$ and $m := \max(g_{01},g_{10},g_{11})$.  We do not exclude the possibility that $m = \infty$.

\begin{proposition}\label{propo:4typesfinite}
Let $G$ be a graph that is not bipartite, let $\sigma$ be a signature on $G$.  Suppose that, for some $ij \neq 00$, $g_{ij} < m$ or $g_{ij} = g_{i^*j^*} = m < \infty$ for some other $i^*j^* \neq 00, ij$.  Then $g_{ij}$ is realized only by cycles.

In particular, if all three types of girth aside from $g_{00}$ are finite, then at least two of the three values  $g_{10}$, $g_{01}$ and $g_{11}$ are realized only by cycles of $G$.
\end{proposition}

\begin{proof}
A \emph{subwalk} of a walk $W$ is any nonempty consecutive sequence of edges in $W$.  A closed walk that realizes some $g_{ij}$ cannot contain a closed subwalk of type $00$, or it would not have minimal length.

Now consider $ij$ as assumed in the statement.  Since $g_{ij} < \infty$, there is a closed walk $W$ that realizes $g_{ij}$. If this walk is not a cycle, by minimality it contains a subwalk $C$ that is a cycle of type $i_Cj_C \neq 00, ij$.  If we cut $C$ out of $W$ we are left with a closed walk $W'$, shorter than $W$ and of type $ij+i_Cj_C = i'j' \neq 00, ij, i_Cj_C$.  Both $C$ and $W'$ are shorter than $W$ so $g_{ij} = m < \infty$ and both $g_{i_Cj_C}$ and $g_{i'j'}$ are less than $m$.  We have shown that either all closed walks that realize $g_{ij}$ are cycles, or $\infty > g_{ij} = m > g_{i_Cj_C}, g_{i'j'}$.  This proves that any $g_{ij} < m$ can only be realized by cycles.

Suppose now that $g_{ij} = m < \infty$.  If $g_{ij}$ is realized by a closed walk $W$ that is not a cycle, the preceding argument shows that $m > g_{i_Cj_C}, g_{i'j'}$ so both $g_{i_Cj_C}$ and $g_{i'j'}$ are realized only by cycles.  Therefore, if any two of $g_{01}, g_{10}, g_{11}$ equal $m$, then all three are realized only by cycles.

The last part of the proposition is now obvious.
\end{proof}

As examples in Figure~\ref{fig:girth-by-closed walks} show, for each choice of the $ij$, $ij\neq 00$, there exist a signed graph $(G,\sigma)$ where the values $g_{ij}(G,\sigma)$, can only be realized only non-cyclic closed walks.

\subsection{Special classes}
\label{sec:special}

Lemma~\ref{lem:2infty} leads to the definition of three classes of signed graph which have proven to be of special importance in the homomorphism study of signed graphs. 
We refer to Section~\ref{sec:Future Work} for examples of important problems in these classes.

\begin{enumerate}
	\item If $g_{11}(G,\sigma)=g_{10}(G,\sigma)=\infty$, then $(G,\sigma)$ has no negative cycle, thus it is balanced. By Lemma~\ref{lem:Zaslavsky}, after a suitable switching, all edges are positive, i.e., $(G,\sigma)$ is switching equivalent to $(G,+)$. This class of signed graphs will be denoted by $\mathcal{G}_{01}$.
	
	\item If $g_{10}(G,\sigma)=g_{01}(G,\sigma)=\infty$, then $(G,\sigma)$ is antibalanced. By Lemma~\ref{lem:Zaslavsky}, after a suitable switching, all edges are negative, i.e., $(G,\sigma)$ is switching equivalent to $(G,-)$. This class of signed graphs will be denoted by $\mathcal{G}_{11}$.
	
	\item If $g_{11}(G,\sigma)=g_{01}(G,\sigma)=\infty$, then $G$ has no cycle of odd length and, thus,
	$(G,\sigma)$ is a signed bipartite graph.  This class of signed graphs will be denoted by $\mathcal{G}_{10}$.  

\end{enumerate}
Thus $\mathcal{G}_{ij}$, $ij \in \set{01,10,11}$, is the class of signed graphs consisting of signed graphs in which every closed walk is either of type $00$ or $ij$. 

Given two signed graphs $(G, \sigma)$ and $(H, \pi)$ both in $\mathcal{G}_{01}$, to decide whether there is a homomorphism of $(G, \sigma)$ to $(H, \pi)$, first we switch each of the graphs so that all edges are positive; then, by Theorem~\ref{thm:HomSecondDefinition}, $(G, \sigma)$ maps to $(H, \pi)$ if and only if $G$ maps to $H$. The same conclusion applies when both graphs are in the class $\mathcal{G}_{11}$. 
While therefore the homomorphism problem in each of these two classes is about homomorphisms of underlying graphs, as signed graphs the two classes behave differently in connection with graph minors.

A \emph{minor} (more precisely, a \emph{link minor}) of a signed graph is a signed graph obtained by the following sequence of operations: 1. Delete vertices and edges. 2. Switch a set of vertices. 3. Contract a set of positive edges.
(A \emph{link} is a non-loop edge.  This definition of contraction, which can be found in \cite{Z97}, is a restriction of the definition of minors from \cite{Z82b}, which also allows contracting loops.  For a positive loop, contraction is the same as deletion.)

In the class $\mathcal{G}_{01}$, the signed graphs without a $(K_3, +)$-minor are the signed graphs with no cycle in the underlying graph (a forest in the usual sense), and thus all admit homomorphisms to $(K_2, +)$. Similarly those not containing $(K_4,+)$ are $K_4$-minor-free graphs together with a signature where all cycles are positive. Thus all such signed graphs map to $(K_3,+)$. 

On the other hand, in the class $\mathcal{G}_{11}$ signed graphs without a $(K_3, -)$-minor are the signed graphs with no odd cycle in the underlying graph. That is because $(K_3,-)$ can only be obtained as a minor from a negative cycle, which in the class $\mathcal{G}_{11}$ is the same as an odd cycle. Thus in this case we have the class of bipartite graphs with signature such that all cycles are positive, and therefore, each admits a homomorphism to $(K_2,-)$. A restatement of a result of Catlin from \cite{C79} (the proof of equivalence is left to the reader) is that a signed graph $(G, \sigma)$ in $\mathcal{G}_{11}$ which does not have a $(K_4, -)$-minor maps to $(K_3, -)$. 

Recalling that the homomorphism question in each of the classes $\mathcal{G}_{01}$ and $\mathcal{G}_{11}$ is (only) about the homomorphisms of the underlying graphs, we may restate previous claims as follows.
Given a graph $G$, if $(G,+)$ has no $(K_3,+)$-minor (respectively, no  $(K_4,+)$-minor), then $G$ maps to $K_2$ (respectively, $K_3$). Similarly, if for a given graph $G$ the signed graph $(G,-)$ has no $(K_3,-)$-minor (respectively, no  $(K_4,-)$-minor), then $G$ maps to $K_3$ (respectively, $K_4$). While the conclusions of the two cases are the same, the class of graphs for which the first statement applies is strictly included in the class of graphs for which the latter one applies.  More importantly, the class of graphs for which the first statement works is a sparse family of graphs whereas the class of graphs in the latter statement in particular contains all bipartite graphs, thus includes a dense family of graphs.  

This observation has been one of the main motivations behind the development of this theory of homomorphisms of signed graphs.  Some possible generalizations will be mentioned in Section~\ref{sec:Future Work}.

\subsection{Definitions of a chromatic number for signed graphs}

One of the most classic notions of graph theory is the notion of \emph{proper coloring} together with the associated parameter: \emph{chromatic number}. Recall that proper coloring of a simple graph is a coloring of vertices where adjacent vertices receive different colors and the chromatic number is the minimum number of colors required in such a coloring. A natural extension of proper coloring of graphs to signed graphs, using signed colors, was first given in \cite{Z82a} (see also \cite{Z82c} and \cite{Z84}). Motivated by this notion of proper coloring, various other definitions of chromatic number of signed graphs have been recently introduced, whose difference depends on the choice of the set of signed colors; see e.g.\ \cite{KS18} for circular coloring and \cite{FW16} for list coloring.

Here we present chromatic number as a natural optimization problem. To better express the idea, we begin with a reformulation of chromatic number of graphs which is better suited to the homomorphism viewpoint: the chromatic number of a simple graph $G$ is the smallest order of a homomorphic image of $G$ which is also simple (i.e., without loops). It is easy to verify that this definition gives the classic chromatic number. 
Observing the importance of the odd cycles in the value of the chromatic number, and that loops, which are odd cycles of length 1, are the crucial objects to be avoided in the target graph, we may refine the definition of chromatic number as follows. Given a triangle-free graph (or more generally one with odd girth $2k+1$) what is the smallest order of a triangle-free graph (or, respectively, a graph of odd girth $2k+1$) to which $G$ admits a homomorphism? 
For example, as an extension of the Four-Color Theorem, and using that theorem itself, it has been shown that every triangle-free planar graph admits a triangle free image of order at most 16 and that 16 is the best possible (see \cite{N13} and references therein, also Section~\ref{sec:Future Work}). This upper bound of 16 is reduced to 5 if we consider the subclass of planar graphs of odd girth at least 13 (see \cite{Zhu01}). 
Considering our no-homomorphism lemma, Lemma~\ref{lem:NoHomGirth}, we generalize this idea to signed graphs as follows.

\begin{definition}
		
	Given a triple  $L=(l_{01}, l_{10}, l_{11})$ where $l_{ij}$ is either a positive integer or infinity, and given a signed graph $(G, \sigma)$ satisfying $g_{ij}(G, \sigma)\geq l_{ij}$, we define the \emph{$L$-chromatic number} of $(G,\sigma)$, denoted by $\chi_L (G,\sigma)$, to be the  minimum number of vertices of a signed graph $(H,\pi)$ which satisfies $g_{ij}(H, \pi)\geq l_{ij}$ and $(G,\sigma)\to (H,\pi)$. 
	
	Furthermore, given a triple $K=(k_{01}, k_{10}, k_{11})$ satisfying $k_{ij}\geq l_{ij}$, we define the $(K,L)$-chromatic number of a class $\cS$ of signed graphs to be the maximum of the $L$-chromatic numbers of signed graphs $(G, \sigma)$ in $\cS$ satisfying $g_{ij}(G,\sigma)\geq k_{ij}$, and $\infty$ when there is no such maximum.
	
\end{definition}

With this terminology, if we take $L=(\infty, \infty, l)$, then any signed graph $(G, \sigma)$ satisfying $g_{ij}(G, \sigma)\geq l_{ij}$ is a signed graph in $\mathcal{G}_{11}$, i.e., an antibalanced signed graph, and thus it can be switched to $(G,-)$ (all edges negative). As mentioned in the previous section, in this class of signed graphs the question of whether $(G, \sigma)$ maps to $(H,\pi)$ is reduced to the question of whether the graph $G$ maps to $H$. Similarly for $L=(l, \infty, \infty)$ a signed graph satisfying $g_{ij}(G, \sigma)\geq l_{ij}$ is a signed graph in $\mathcal{G}_{01}$, i.e., a balanced signed graph, which switches to $(G,+)$, so homomorphism questions on these signed graphs are identical to homomorphism questions on graphs.

Let $\mathcal{P}$ be the class of signed planar graphs.  With our terminology, it is a restatement of the Four-Color Theorem to say that for $K=L=(\infty, \infty, 3)$ and for $K=L=(3, \infty, \infty)$ the $(K,L)$-chromatic number of $\mathcal{P}$ is 4. For $K=(\infty, \infty, 5)$ and $L=(\infty, \infty, 3)$ it is a restatement of the Gr\"otzsch theorem to say that the $(K,L)$-chromatic number of $\mathcal{P}$ is 3. The result of \cite{N13} can be also restated as: for $K=L=(\infty, \infty, 5)$ the $(K,L)$-chromatic number of $\mathcal{P}$ is $16$.  In general, determining the $(K,L)$-chromatic number of planar graphs is a question of high interest. Special cases will be mentioned in Section~\ref{sec:Future Work}.

The notion of proper coloring of signed graphs as defined in \cite{Z82a} corresponds to the $(K,L)$-chromatic number for $K=L=(3, 2, 1)$.  That is to say, given a signed graph where digons and negative loops are allowed but positive loops are not, we want to find a smallest homomorphic image without a positive loop. For more detail we refer to Section 2.4 of \cite{BFHN17}.

\section{Connections to 2-Edge-Colored Graphs}\label{Sec:Connection}

A homomorphism of a 2-edge-colored graph $G=(V(G), E_1(G), E_2(G))$ to a 2-edge-colored graph $H=(V(H), E_1(H), E_2(H))$ is a homomorphism of the underlying graph $G$ to the underlying graph $H$ which preserves colors of the edges.

Homomorphisms of signed graphs can be viewed as a special case of homomorphisms of 2-edge-colored graphs in a few ways; we discuss three such possibilities here.

\subsection{Signs as colors}

The easiest connection is by way of Theorem~\ref{thm:HomSecondDefinition}. A signed graph $(G,\sigma)$ is a 2-edge-colored graph with the colors $+$ and $-$. Then an edge-sign-preserving homomorphism of signed graphs is identical to a homomorphism of 2-edge-colored graphs. 
Thus by Theorem~\ref{thm:HomSecondDefinition}, there is a (switching) homomorphism of $(G,\sigma)$ to $(H, \pi)$ if and only if there are a switching $\sigma'$ of $\sigma$ and a 2-edge-colored graph homomorphism $(G,\sigma') \to (H, \pi)$.

\subsection{Double switching graph}

The second connection is based on a construction first presented by R.~C.~Brewster and T.~Graves in \cite{BG09}. In fact, in that paper Brewster and Graves introduced and studied (independently of other writers) the notion of homomorphisms of signed graphs as ``edge-switching [i.e., switching] homomorphisms of edge-colored graphs''. 

\newcommand\DSG{\operatorname{DSG}}
\newcommand\USG{\operatorname{USG}}

\begin{definition}
Let $(G,\sigma)$ be a signed graph. We define the \emph{double switching graph} $\DSG(G,\sigma)$ to be a signed graph with two vertex sets, $V^+$ and $V^-$, each a copy of $V(G)$, with edges $u^\alpha v^\beta$ ($\alpha, \beta = \pm$) for each edge $uv$, and signs $\bar\sigma(u^\alpha v^\beta) = \alpha\beta\sigma(uv)$. 
We normally treat $\DSG(G,\sigma)$ as a 2-edge-colored graph.
\end{definition}

Observe that $(G,\sigma)$ is an induced subgraph of $\DSG(G,\sigma)$. It is induced by $V^+$ and also by $V^-$. To switch a vertex $x$ of signed graph $(G,\sigma)$, when viewed as a subgraph induced by $V^+$, is to replace $x^+$ by $x^-$. Thus $\DSG(G,\sigma)$ can be partitioned into two copies of $(G,\sigma')$ for any switching $\sigma'$ of $\sigma$. For this reason one may refer to $\DSG(G,\sigma)$ as the \emph{double switching graph} of $(G,\sigma)$.

Brewster and Graves used this construction to connect homomorphisms of signed graphs to homomorphisms of 2-edge-colored graphs. We restate here two main connections:
 
 \begin{theorem}\label{Thm:HomSig->Hom2EdgeColored}
 	Given signed graphs $(G,\sigma)$ and $(H, \pi)$, there exists a switching homomorphism of  $(G,\sigma)$ to $(H, \pi)$ if and only if there exists a color-preserving homomorphism of the 2-edge-colored graph $(G,\sigma)$ (equivalently, of $\DSG(G,\sigma)$) to the 2-edge-colored graph $\DSG(H, \pi)$. 
 \end{theorem}
 
\begin{theorem}
	A signed graph $(G,\sigma)$ is a core if and only if $\DSG(G,\sigma)$ is a core as a 2-edge-colored graph.
\end{theorem}

It is noteworthy that if $(G,\sigma)$ is in $\mathcal{G}_{ij}$ for $i, j \in \{ 0,1 \}$, then so is $\DSG(G,\sigma)$ when viewed as a signed graph.

\subsection{Extended double cover}

We now introduce a new construction which provides a different setting to capture homomorphisms of signed graphs as a special case of homomorphisms of 2-edge-colored graphs.

\begin{definition}
	Let $(G,\sigma)$ be a signed graph. We define $\EDC(G,\sigma)$ to be a signed graph on vertex set $V^+ \cup V^-$, where $V^+ := \{v^+: v \in V(G)\}$ and $V^- := \{v^-: v \in V(G)\}$.
	Vertices $x^+$ and $x^-$ are connected by a negative edge; all other edges, to be described next, are positive. 
	If vertices $u$ and $v$ are adjacent in $(G,\sigma)$ by a positive edge, then $v^+u^+$ and $v^-u^-$ are two positive edges of $\EDC(G,\sigma)$, if vertices $u$ and $v$ are adjacent in $(G,\sigma)$ by a negative edge, then $v^+u^-$ and $v^-u^+$ are two positive edges of $\EDC(G,\sigma)$.  
\end{definition}

Since $\EDC(G,\sigma)$ consists of the double covering graph of $(G,\sigma)$, as defined in \cite{Z82b}, with positive edge signs and added negative edges $x^+x^-$, we call it the \emph{extended double cover} of $(G,\sigma)$.  Let $(G,\sigma)^\circ$ denote $(G,\sigma)$ with a negative loop attached to each vertex.  There is a natural projection $p_G: \EDC(G,\sigma) \to (G,\sigma)^\circ$ induced by mapping $x^+,x^- \mapsto x$ for each $x \in V(G)$.  The negative edges map to the negative loops.  A positive edge $u^\alpha v^\beta$ maps to the edge $uv$ with sign $\alpha\beta$.  A \emph{fiber} of $p$ is any set $p_G\inv(u) = \{u^+,u^-\}$ (for a vertex) or $p_G\inv(uv) = \{u^+v^{\sigma(uv)}, u^-v^{-\sigma(uv)}\}$ (for an edge), which for a negative loop means $p_G\inv(uu) = \{u^+u^-\}$.

To be of interest for homomorphisms we first show that $\EDC(G,\sigma)$ is independent of switching.  
We define a \emph{fibered automorphism} of $\EDC(G,\sigma)$ to be an automorphism that preserves fibers; that is, it carries an element of a fiber to itself or to the other element of the same fiber.  
More generally, we define a \emph{fibered homomorphism} $\psi: \EDC(G,\sigma) \to \EDC(H,\pi)$ to be a graph homomorphism that respects fibers; that is, if $x\in V(G)$, then $\{\psi(x^+),\psi(x^-)\} = \{y^+,y^-\}$ for some $y \in V(H)$.  This definition implies that $\psi(x^+x^-)=y^+y^-$ and that $\psi$ carries an edge fiber $p_G\inv(uv)$ to $p_H(yz)$ for some edge $yz \in E(H)$.

\begin{lemma}\label{lem:fiberedmorphism}
Let $(G,\sigma)$ and $(H,\pi)$ be signed graphs.  Suppose $\psi$ is a graph homomorphism $\EDC(G,\sigma) \to \EDC(H,\pi)$.  Then $\psi$ is a 2-edge-colored homomorphism (i.e., it preserves edge signs) if and only if it is a fibered homomorphism.
\end{lemma}

\begin{proof}
The proof follows from the fact that the only negative edges are of the form $x^+x^-$ and such an edge exists for every vertex $x$ in both graphs.
\end{proof}

\begin{proposition}\label{prop:R-Indpendent-Of-Sigma}
	Let $(G,\sigma)$ be a signed graph and let $\sigma'$ be a switching of $\sigma$. Then $\EDC(G,\sigma)$ is isomorphic to $\EDC(G,\sigma')$ as a 2-edge-colored graph.  The isomorphism is a fibered isomorphism.
\end{proposition} 

\begin{proof}
	Let $X$ be a set of vertices that are switched in order to get $\sigma'$ from $\sigma$. An isomorphism of  $\EDC(G,\sigma)$ to $\EDC(G,\sigma')$ then consists of interchanging the roles of $v^+$ and $v^-$ for every vertex $v\in X$.  This is a fibered isomorphism.
\end{proof}

\begin{theorem}\label{thm:HomByEDC(G,sigm)}
	There is a homomorphism of a signed graph $(G,\sigma)$ to a signed graph $(H, \pi)$ if and only if there is a color-preserving homomorphism of $\EDC(G,\sigma)$ to $\EDC(H,\pi)$.
\end{theorem}

\begin{proof}
	Suppose $(G,\sigma)$ maps to $(H, \pi)$ as a signed graph. The by Theorem~\ref{thm:HomSecondDefinition} there are a switching $\sigma'$ of $\sigma$ and a mapping $\phi$ of $(G,\sigma')$ to $(H, \pi)$ which preserves the signs as  well. By Proposition~\ref{prop:R-Indpendent-Of-Sigma} switching $\sigma$ will give an isomorphic copy of $\EDC(G,\sigma)$ (as a 2-edge-colored graph), thus, without loss of generality, $\sigma'=\sigma$.
	
	We need to present a homomorphism from $\EDC(G,\sigma)$ to  $\EDC(H,\pi)$ which preserves edge signs (i.e., colors). We show that the most natural extension of $\phi$ works. More precisely, let $\psi(x^ \epsilon)=\phi(x)^ \epsilon $ for each vertex $x$ of $G$ and for $\epsilon =+,-$ (this gives the mapping of vertices); for the edges we define $\psi(x^\epsilon y^\delta)=\phi(x)^\epsilon\phi(y)^\delta$ for every edge $x^\epsilon y^\delta$ of $\EDC(G,\sigma)$, where $\epsilon, \delta=+,-$ (this gives the mapping of edges).  This is a homomorphism because $\phi$ never maps two adjacent vertices of $G$ to a single vertex of $H$. Also, $\psi$ is clearly sign-preserving.
	
	For the other direction, suppose $\psi$ is a homomorphism from $\EDC(G,\sigma)$ to $\EDC(H, \pi)$ preserving signs. Thus each pair $x^+x^-$ must map to a pair of the form $z^+z^-$ as these are the only negative edges (i.e., $\psi$ is a fibered homomorphism). To complete the proof we should introduce a switching $\sigma'$ of $\sigma$ under which we may find a homomorphism preserving edge signs. To this end it is enough to decide for each vertex $x$ whether we switch $x$. As $x^+x^-$ is mapped to a pair $z^+z^-$, either we have $\psi(x^+)=z^+$ and $\psi(x^-)=z^-$ or we have $\psi(x^+)=z^-$ and $\psi(x^-)=z^+$.  In the former case we do not switch $x$, in the latter case we do switch $x$. Let $\sigma'$ be the signature obtained after such switching. Observe that by the definition of $\sigma'$, after composing the mapping $\psi$ with the isomorphism which changes $\EDC(G,\sigma')$ to $\EDC(G,\sigma)$ (by Proposition~\ref{prop:R-Indpendent-Of-Sigma}) we have a homomorphism $\psi'$ of $\EDC(G,\sigma')$ to $\EDC(H, \pi)$ which preserves vertex signs, i.e., $\phi'(x^\epsilon) = z^\epsilon$ for some $z \in V(H)$.
	
	We now define $\phi$ as a sign preserving homomorphism from $(G,\sigma')$ to $(H, \pi)$ as follows:
	for the vertex mapping, if $\psi(x^+x^-)=z^+z^-$ then define $\phi(x)=z$; for the edge mapping, let $p_G\psi(xy) = \{x^+y^\epsilon, x^-y^\epsilon\}$ in $\EDC(G,\sigma')$.  Thus $\sigma'(xy) = \epsilon$ and since $\psi'$ preserves vertex signs it is straightforward to show that $\pi(\phi(xy)) = \epsilon = \sigma'(xy)$, as we wished.
\end{proof}	

It worth noting the effect of the extended double cover construction on the three special classes of signed graphs. If $(G, \sigma)$ is balanced (that is to say $(G, \sigma)$ is in the class $\mathcal{G}_{01}$), then one may switch $\sigma$ to $\sigma'$ so that all edges are positive. Thus $\EDC(G,\sigma')$ consists of two vertex-disjoint copies of $G$ with a negative edge between each pair of corresponding vertices; we observe that $\EDC(G,\sigma')$ itself is a balanced signed graph and so is $\EDC(G,\sigma)$ (by Proposition \ref{prop:R-Indpendent-Of-Sigma}). If $(G,\sigma)$ is antibalanced, then $\EDC(G,\sigma)$ is a signed bipartite graph, and vice versa, if $(G,\sigma)$ is a signed bipartite graph, then $\EDC(G,\sigma)$ is antibalanced. We prove this and more in the next proposition.

\begin{theorem}\label{prop:AntiBalance<-->Bipartite}  Let $(G,\sigma)$ be a signed graph.
\begin{enumerate}[{\rm(1)}]
\item	$(G,\sigma)$ is antibalanced if and only if  $\EDC(G,\sigma)$ is a signed bipartite graph.
\item	$(G,\sigma)$ is a signed bipartite graph if and only if  $\EDC(G,\sigma)$ is antibalanced.
\end{enumerate}	
\end{theorem}

\begin{proof}
(1)	If $(G,\sigma)\in \mathcal{G}_{11}$, then it can be switched to $(G, -)$ such that all edges are negative. It is then immediate that $\EDC(G, -)$ is a signed bipartite graph, and it is isomorphic to $\EDC(G,\sigma)$.

	Conversely, suppose $\EDC(G,\sigma)$ is a signed bipartite graph.  Each two vertices $x^+,x^-$ must be in opposite sets of the bipartition, so by switching to, say, $(G,\sigma')$ we can ensure that the two sets are $V^+$ and $V^-$.  Then every edge in $\EDC(G,\sigma)$ has the form $x^\epsilon y^{-\epsilon}$, so every edge of $(G,\sigma')$ is negative.  It follows that $(G,\sigma)$ is antibalanced.
	
(2)	We now consider a signed bipartite graph $(G,\sigma)$. Let $A,B$ be the partition of vertices of $G$ into two independent sets. By switching $A^+ \cup A^-$ we make all edges of $\EDC(G,\sigma)$ negative.  Thus, $\EDC(G,\sigma)$ is antibalanced.

	For the converse, suppose $\EDC(G,\sigma)$ is antibalanced.  That means we can make it all negative by switching a set $X \subseteq V^+\cup V^-$.  Since $x^+x^-$ is already negative, $x^+$ and $x^-$ are either both in $X$ or both not in $X$.  For two adjacent vertices $x,y \in V(G)$, exactly one of $x^+$ and $y^+$ is in $X$.  Let $A := p_G(X)$ and $B := V(G) \setminus A$; this gives a bipartition of $V(G)$ such that no edge of $G$ has both endpoints in $A$ or both in $B$.  Thus, $G$ is bipartite.
\end{proof}

The extended double cover can be used to give a nice inductive definition of signed projective cubes. These graphs and problems about homomorphism to them are what motivated the definition of a homomorphism of signed graphs in the first place (see \cite{NRS13}). We write more in Section~\ref{sec:Future Work} on the importance of these graphs and a connection to the Four-Color Theorem.

\begin{definition}\label{def:spc}
	A \emph{signed projective cube} of dimension 1, denoted by $\SPC(1)$, is a signed graph on two vertices connected by one negative edge and one positive edge. A signed projective cube of dimension $k$, $k\geq 2$, denoted by $\SPC(k)$, is defined to be $\EDC(\SPC(k-1))$.
\end{definition}

\begin{proposition}\label{prop:spc}
For $k\geq1$, $\SPC(k)$ is isomorphic as a $2$-edge-colored graph to the signed graph whose vertices are the elements of $\mathbb{Z}_2^{k}$ and in which two vertices $x$ and $y$ are adjacent by a positive edge if they differ in exactly one coordinate and are adjacent by a negative edge if they differ in all coordinates.
\end{proposition}

That is, $\SPC(k)$ is isomorphic as a $2$-edge-colored graph to the signed graph that consists of a $k$-dimensional cube graph $Q_k$ on vertex set $\{0,1\}^{k}$ with all edges positive and with additional negative diagonals $x\bar{x}$, where $\bar{x} := (1,1,\ldots,1)-x$.

\begin{proof}
The proposition is obviously true for $k=1$.  We assume it for $k$ and show how the construction of $\SPC(k+1)$ in Definition \ref{def:spc} produces the graph in the proposition.

By definition, the vertex set of $\EDC(\SPC(k+1))$ is $\{(x,0),(x,1): x \in \{0,1\}^k \}$ and the edges are the positive edges $(x,0)(y,0)$ and $(x,1)(y,1)$ whenever $xy \in E^+(\SPC(k))$, $(x,0)(\bar{y},1)$ and $(x,1)(\bar{y},0)$ whenever $xy \in E^-(\SPC(k))$ (which means that $y = \bar{x}$), and the negative edges $(x,0)(x,1)$ for all $x \in \{0,1\}^k$.  Now relabel every vertex $(x,1)$ by the new name $(\bar{x},1)$.  With the new labeling, the signed edge set is identical to that of the definition of $\SPC(k+1)$ in the proposition.
\end{proof}

\begin{corollary}\label{cor:spc-anti-bipart}
The signed projective cube $\SPC(k)$ is antibalanced if and only if $k$ is even and is a signed bipartite graph if and only if $k$ is odd.
\end{corollary}

\begin{proof}
This follows by induction from Theorem \ref{prop:AntiBalance<-->Bipartite} applied to the definition of $\SPC(k)$.
\end{proof}


\section{Further Discussion and Future Work}\label{sec:Future Work}

We expect to extend many properties of homomorphisms and coloring of graphs to signed graphs. The special interest here will be a stronger connection with graph minors. However, many interesting ideas seem to arise independently. Here we mention only some of the most challenging problems of the subject. 

\subsection{Signed projective cubes and packing of negative edge sets}

The problem that motivated the introduction of homomorphisms of signed graphs (by B. Guenin) is the following conjecture in extension of the Four-Color Theorem, proposed in \cite{N07} and \cite{G05}:

\begin{conjecture}\label{conj:MapingToSPC} We propose:
\begin{enumerate}[{\rm1.}]
	\item 	If $(G,-)$ is an antibalanced signed planar graph satisfying $g_{11}(G, -)\geq 2k+1$, then $(G, -)\to \SPC(2k)$.
	\item 	If $(G, \sigma)$ is a signed planar bipartite graph satisfying $g_{10}(G, \sigma)\geq 2k$, then $(G, -)\to \SPC(2k-1)$.
\end{enumerate}
\end{conjecture}

Combined with a result of \cite{NSS16} and terminology introduced here the conjecture amounts to claiming that 
\begin{enumerate}[i.]
\item for $K=L=(\infty, \infty, 2k+1)$ the $(K,L)$-chromatic number of planar graphs is $2^{2k}$, 
\item for $K=L=(\infty, 2k, \infty,)$ the $(K,L)$-chromatic number of planar graphs is $2^{2k-1}$. 
\end{enumerate}

It is then natural to ask more generally for the $(K,L)$ chromatic number of planar graphs, for each value of $K$ and $L$. For $K=(3,4,3)$, and $L=(3,2,1)$ it was proposed by Máčajová, Raspaud and Škoviera \cite{MRS16} that the $(K,L)$-chromatic number of the class of planar signed graphs is 2. This conjecture has recently been disproved in \cite{KN19}.

The question captures some other well-known results or conjectures in the theory of coloring planar graphs, such as the Gr\"otzsch theorem and the Jaeger--Zhang conjecture. 

Conjecture~\ref{conj:MapingToSPC} is related to several other conjectures, a notable one being a conjecture of P. Seymour on the edge-chromatic number of planar multigraphs. The connection with homomorphisms to projective cubes and edge-partitioning problems is stated in the following theorem, we refer to \cite{NRS13} for a proof and further details on connections to Seymour's edge-coloring conjecture.

\begin{theorem} [\cite{NRS13}]\label{thm:MappingToSPC-EdgeColoring}
	A signed graph $(G,\sigma)$ maps to $\SPC(k)$ if and only if $E(G)$ can be partitioned into $k+1$ sets $E_i$ such that each $E_i$ is the set of negative edges of a signature that is switching equivalent to $\sigma$. 	
\end{theorem}

This leads to the following packing problem:

\begin{problem}\label{prob:PackingSignatures}
Given a signed graph $(G, \sigma)$ what is the maximum number of signatures $\sigma_1, \sigma_2, \ldots, \sigma_l$ such that
\begin{itemize}
\item[i.] each $\sigma_i$ is switching equivalent to $\sigma$, and
\item[ii.] the sets $E^-(\sigma_i)$ are pairwise disjoint? 
\end{itemize}	
\end{problem}

A starting point for this question then is to ask whether, given a signed graph $(G, \sigma)$, there is a switching $(G, \sigma')$ that has no negative edge in common with $(G, \sigma)$. 
Using our results from Section~\ref{sec:CycleSpace} we may provide an easy test for this case of the question.  We consider that the empty graph, with no vertices or edges, is bipartite.

\begin{theorem}\label{thm:SeparateSignature}
Given a signed graph $(G,\sigma)$, there exists a switching $(G, \sigma')$ such that the two signed graphs have no negative edge in common if and only if the set of negative edges of $(G, \sigma)$ induces a bipartite graph.
\end{theorem}

\begin{proof}
Let $G^-$ be the subgraph induced by the set of negative edges of $(G,\sigma)$.  Note that $G^-$ is bipartite if and only if its edge set is a subset of a cut $[X,Y]$ with $Y:=V(G)\setminus X$.  

	For sufficiency assume that $G^-$ is bipartite.  Switching $X$, the negative edge set becomes $[X,Y] \setminus E^-$, which is disjoint from $E^-$ and induces a bipartite subgraph.  

	Conversely, if $(G,\sigma)$ and $(G, \sigma')$ are switching equivalent with no common negative edge, then the cut $[X,Y]$ which we switch to get one from the other contains all negative edges of both; thus, both sets of negative edges together induce a bipartite graph.

	(This short proof was contributed by Nicholas Lacasse.)
\end{proof}	

We note that the bipartition $\{X,Y\}$ may not be unique, so the switched signed graph may not be unique, but each such resulting signed graph has negative edge set that is disjoint from $E^-$.  Indeed, every $(G,\sigma')$ whose negative edge set is disjoint from $E^-$ is obtained in this way.  

\begin{proof}[Alternate Proofs]
Here is a different approach to necessity.  
	If $G^-$ is not bipartite, then it contains an odd cycle $C$. Since all edges of $C$ are negative and it has odd length, it is a negative cycle; thus any signature which is switching equivalent to $\sigma$ must have an odd number of edges of $C$; in particular at least one negative edge must belong to $C$, thus to $G^-$.
	
Here is an entirely different proof of sufficiency using a contraction technique that has other applications.  As in the short proof above, we assume $G^-$ is bipartite.  Thus each negative cycle, and, therefore, each negative closed walk, has at least one positive edge. 
	Let $\cW$  be the set of negative closed walks of $G$. By Theorem~\ref{thm:3-walks}, $\cW$ satisfies the exclusive 3-walk property. Let $\bar{G}$ be the minor of $G$ obtained by contracting all negative edges of $(G, \sigma)$ (without deletion of parallel edges or loops), and let $\bar{\cW}$ be the set of closed walks in $\bar{G}$ obtained from walks in $\cW$. As each walk in $\cW$ has at least one positive edge, none of the closed walks in $\bar{\cW}$ is a trivial walk. It is then straightforward to check that $\bar{\cW}$ has inherited the exclusive 3-walk property, and therefore, by Theorem~\ref{thm:3-walks},  $\bar{\cW}$ is the set of negative closed walks $(\bar{G}, \eta)$ for some signature $\eta$ of $\bar{G}$. As each edge of $\bar{G}$ corresponds to an edge of $G$, we may lift $\eta$ to a signature $\sigma'$ of $G$ which assigns negative sign to exactly those edges to which $\eta$ has assigned a negative sign.
	As such, the set of negative closed walks of $(G,\sigma')$ is the same as that of $(G, \sigma)$. But since all negative edges of $\sigma$ were contracted to get $\bar{G}$, the two signed graphs on $G$ have no negative edge in common.
\end{proof}

We would like to point out the importance of this theorem in approaching Conjecture~\ref{conj:MapingToSPC}. To this end we first restate the conjecture in the following uniform form:

\begin{conjecture}\label{conj:MapingToSPC-restated}
	A planar signed graph $(G,\sigma)$ maps to $\SPC(k)$ if and only if the four girth conditions of Lemma~\ref{lem:NoHomGirth}, the no-homomorphism lemma, are satisfied.   
\end{conjecture}

The two conjectures are equivalent because if for $(G,\sigma)$ we have $g_{ij}(G, \sigma) \geq g_{ij}(\SPC(2k))$, then since $g_{01}(\SPC(2k))=g_{10}(\SPC(2k))=\infty$, $(G,\sigma)$ must be an antibalanced graph; and similarly $(G, \sigma)$ must be a signed bipartite graph in order to satisfy $g_{ij}(G, \sigma) \geq g_{ij}(\SPC(2k-1))$.

While the case $k=1$ of the conjecture is rather trivial, the case $k=2$ of the conjecture is equivalent to the four-color theorem. It is verified for $k\leq 7$ through an equivalent reformulation and by means of induction on $k$ (we refer to \cite{NRS13} for details). In attempt to apply induction (based on $k$) directly on this statement, 
what one would need is to find a signature $\sigma'$ equivalent to $\sigma$, such that the signed graph $(\tilde{G}, \eta)$ obtained by switching as in the alternate proof of Theorem~\ref{thm:SeparateSignature} satisfies the four conditions of the no-homomorphism lemma for mapping to $\SPC(k-1)$. 

One may also consider a more general form of the question.

\begin{problem}\label{prob:disjointnegsets}
	Given signed graphs $(G,\sigma)$ and $(G, \eta)$, are there switchings $(G, \sigma')$ and $(G, \eta')$ such that the two  signatures have no negative edge in common?
\end{problem}

Problem \ref{prob:PackingSignatures} is the planar dual of the packing problem for $T$-joins.  Consider a signed plane graph $(G,\sigma)$ and the consequent signs of dual vertices (i.e., faces of $G$).  The dual edges have the same signs as in the original signed graph.  Let $T$ be the set of negative dual vertices.  The negative edges of the dual graph $G^*$ have the property that they have even degree at every positive vertex of $G^*$ and odd degree at every negative vertex.  This is the definition of a $T$-join.  Switching $(G,\sigma)$ to $(G,\sigma')$ does not change $T$, but it does change the negative edge set.  Thus, $E^-(G,\sigma)$ and $E^-(G,\sigma')$ are disjoint if and only if the corresponding $T$-joins are disjoint.  Thus, packing the most negative edge sets as in Problem \ref{prob:PackingSignatures} is the planar dual of packing the most $T$-joins in a graph with a fixed set $T$ of evenly many vertices.

Finally, we would like to mention a possible extension of Conjecture~\ref{conj:MapingToSPC-restated} in relation to minors theory.

\begin{conjecture}\label{conj:MapingToSPC-MinorVersion}
	A signed graph $(G,\sigma)$ without a $(K_5, -)$-minor maps to $\SPC(k)$ if and only if the four girth conditions of Lemma~\ref{lem:NoHomGirth}, the no-homomorphism lemma, are satisfied.   
\end{conjecture}

\subsection{Future work}

One may observe that the four no-homomorphism conditions of Lemma~\ref{lem:NoHomGirth} are on the four possible types of cycle based on the parity of the length and the sign. Parity itself is the sign of the cycle when all edges are negative. This leads to the idea of assigning an array of signatures to a given graph. A homomorphism among such structures then is a mapping of basic elements that preserves the signs of closed walks in all signatures. We hope to address this generalization and its properties in future work.

The algorithm of Proposition~\ref{prop:Computing g_ij} is designed only to show that $g_{ij}(G, \sigma)$ can be computed in polynomial time. It would of interest to design more efficient algorithms.   We believe the computation should be possible in time $O(n^3)$.

As an extension of the four-color problem, Hadwiger proposed that every simple graph with no $K_n$-minor admits a proper $(n-1)$-coloring. That is, any balanced signed graph without a loop and with no $(K_n, +)$-minor admits a homomorphism to the balanced complete graph on $n-1$ vertices. This is widely believed to be one of the most challenging problems in graph theory. B. Gerards and P. Seymour independently then introduced a strengthening of the conjecture which can be well stated using the terminology of signed graphs:

\begin{conjecture} [Odd-Hadwiger]
	If $(G,-)$ is an antibalanced signed graph with no loop that has no $(K_n,-)$-minor, then $(G, -)\to (K_{n-1},-)$.  
\end{conjecture}  

This important conjecture remains wide open. B. Guenin in 2005 claimed a proof for $n=4$ (assuming the four-color theorem), but to now there has been no publication.


\section{Acknowledgement}
\bigskip

\noindent
{\bf Acknowledgement.}
This work is supported by the French ANR project HOSIGRA (ANR-17-CE40-0022). The work is descendant of an earlier work with Cl\'ement Charpentier and has benefited from discussions with him.

\bigskip

\bibliographystyle{plain}

\end{document}